\documentclass[12pt,twoside,reqno]{amsart}
\usepackage{graphicx} 
\usepackage[T1]{fontenc}
\usepackage[utf8x]{inputenc}
\usepackage[english]{babel}
\usepackage{amsmath}
\usepackage{amsthm}
\usepackage{tensor}
\usepackage{amssymb}
\usepackage{bbm}   
\usepackage{fancyhdr}               
\usepackage{fixltx2e}               
\usepackage{enumitem}               
\usepackage[all,cmtip]{xy}          
\usepackage{graphicx}
\usepackage{esint} 
\usepackage{mathrsfs} 
\usepackage{faktor}
\usepackage{mathtools}
\newtheorem{theorem}{Theorem}            

\newtheorem{lemma}[theorem]{Lemma}

\theoremstyle{definition}              

\theoremstyle{remark}                  

\newcommand{\R}{\mathbb{R}}                 

\newcommand{\eps}{\varepsilon}

\renewcommand{\exp}{\textup{exp}}

\DeclareMathOperator{\curl}{\curl}                                   

\usepackage{color}
\usepackage[normalem]{ulem} 

\numberwithin{equation}{section}
\numberwithin{definition}{section}
\numberwithin{theorem}{section}
\numberwithin{remark}{section}

\begin{document}
\title{On Variational Approximations For Wave Maps}
\author{Zhiyuan Geng, Changyou Wang}
\address{Department of Mathematics, Purdue University, West Lafayette, IN 47907}
\email{geng42@purdue.edu, wang2482@purdue.edu}
\begin{abstract} In this paper, we revisit the existence of global weak solutions of wave maps from $\R^n$ into the sphere $\mathbb{S}^{L-1}$, $\Box u\perp T_u \mathbb{S}^{L-1}$, by establishing it as a singular limit of maps from
$\R^n\times \R_+$ to $\mathbb S^{L-1}$ that minimize elliptic regularized variational functionals that contain an exponential weight in the time direction with a small parameter $\varepsilon$, where the initial data of the Cauchy problem serve as the boundary condition. The idea went back to De Giorgi \cite{Giorgi1996}, which has been implemented by Serra and Tilli \cite{Serra-Tilli2012, Serra-Tilli2016}
for certain class of nonlinear wave equations. This approach 
is also applicable to the $SO(m)$-target manifold. 
\end{abstract}
\date{\today}
\maketitle

\centerline{\it For the Commemoration of Professor Wei-Yue Ding}

\bigskip
\section{Introduction}

Recall that for $n, k\ge 1$ and a given a smooth, $k$-dimensional compact Riemannian manifold $(N,h)\hookrightarrow\mathbb{R}^L$
(isometrically embedded into
the euclidean space $\R^L$ for a sufficiently lareg $L$), 
a map $u:\R^n\times (0,\infty)\to N\hookrightarrow \R^L$ is called a wave map, if it solves the wave map equation:
\begin{align}
&\Box u:=\partial_{tt}u-\Delta u\perp T_u N \ 
\ {\rm{in}}\ \ \R^n\times (0,\infty),\label{wavemap1}\\
&(u, \partial_t u)\big|_{t=0}=(\phi,\psi)\
\ {\rm{in}}\ \ \R^n, \label{cauchydata}
\end{align}
where $\phi:\R^n\to N$ and $\psi: \R^n\to \phi^*TN$, the induced tangent bundle $TN$ by $\phi$, are given initial data. Observe that a wave map $u$ can be viewed as a harmonic map from $(M, g_0)=(\R^n\times (0,\infty), g_0)$ to $N$, 
with the Minkowski metric given by $g_0=dx^2-dt^2$, that is, 
$u:\R^n\times (0,\infty)\to N$ is a critical point of the  functional
\[
H(u)=\iint_{\R^n\times (0,\infty)}
(|\partial_t u|^2-|\nabla u|^2)\,dxdt.
\]
However, since the functional $H$ is neither convex nor coercive, it is impossible to find a wave map through the method of calculus of variations on $H$. When $n\ge 2$ and $N=\mathbb{S}^{L-1}\subset\R^L$ is the unit sphere, Shatah \cite{Shatah1988} was the first to establish the existence of global weak solutions of \eqref{wavemap1}-\eqref{cauchydata} for $(\phi,\psi)\in \dot{H}^1(\R^n,\mathbb S^{L-1})\times L^2(\R^n, \phi^*T\mathbb{S}^{L-1})$ through
approximation by semi-linear wave equations, namely, the Ginzburg-Landau approximation. Later,  Freire \cite{Freire1996} and Zhou \cite{Zhou1999}
extended the global existence of weak solutions of \eqref{wavemap1}-\eqref{cauchydata} to Riemannian homogeneous spaces by Ginzburg-Landau approximations and viscous approximations of wave maps respectively. For general target manifolds $N$,  in a very important work \cite{FreireMullerStruwe1997} (cf. also \cite{FreireMullerStruwe1998}) Freire-M\"uller-Struwe successfully
established, in dimension $n=2$, the compensated-compactness property for weakly convergent sequences of wave maps by adapting
H\'elein's method of moving frames on $N$ and the duality
between Hardy and BMO spaces for harmonic maps. Utilizing
\cite{FreireMullerStruwe1997} and Zhou's method of viscous approximations, M\"uller-Struwe \cite{MullerStruwe1996} obtained the global existence
of weak wave maps into general manifolds when $n=2$.
It has remained an open problem that  whether there exists
global weak weak solutions of \eqref{wavemap1}-\eqref{cauchydata}
for general smooth target manifolds $N$, when $n\ge 3$. For interested readers, we refer to the book \cite{Shatah-Struwe} by Shatah and Struwe that contains many interesting development on geometric wave equations. 

In this article, we revisit the wave map equation \eqref{wavemap1} by investigating the elliptic regularization approach proposed by De Giorgi \cite{Giorgi1996} and Serra-Tilli \cite{Serra-Tilli2012, Serra-Tilli2016}.

Now, we describe the approach. For $\eps>0$, define a family of
energy functionals
\begin{align}
{E}_\eps(v)=\iint_{\mathbb R^n\times\R_+} \frac{e^{-t/\eps}}{\eps}
\Big\{\eps|\partial_t^2v|^2+\frac{1}{\eps}|\nabla v|^2\Big\}\,dxdt,
\end{align}
for maps $v:\mathbb{R}^n\times (0,\infty)\to N$ subject to the boundary condition:
\begin{align}
v(x,0)=\phi(x), \ \ v_t(x,0)=\psi(x), 
\ \forall x\in\mathbb{R}^n,
\end{align}
where $\phi:\mathbb{R}^n\to N$
and $\psi:\mathbb{R}^n\to \mathbb{R}^L$,
with $\psi(x)\in T_{\phi(x)}N$ for $x\in\mathbb{R}^n$, are given data.

By calculating the first order variation, any minimizer $v_\eps:
\R^n\times (0,\infty)\to N$ for ${E}_\eps$, provided it exists in suitable configuration function spaces,  
solves the following Euler-Lagrange equation
\begin{align}\label{WED1}
\eps^2 \partial_{t}^4v_\eps-2\eps \partial_t^3v_\eps+
\partial_t^2v_\epsilon-\Delta v_\eps
\perp T_{v_\eps} N.
\end{align}
Heuristically, as $\eps\to 0$, a subsequence of solutions $v_\eps$ of \eqref{WED1} would converge to a map $v:\R^n\times (0,\infty)\to N$, that satisfies the wave map equation \eqref{wavemap1}. On the other hand, how to justify the validity of such a heuristics turns out to be a very challenging problem, thanks to the supercritical nonlinearities appearing within the corresponding analytic form of the equation \eqref{WED1}.

Neverthless, when $N=\mathbb{S}^{L-1}\hookrightarrow\mathbb R^L$ or 
the special orthogonal group $SO(m)\hookrightarrow\R^L$ 
(with $L={m^2}$),
we are able to prove the following theorem.

\begin{theorem}  \label{main} Assume $N=\mathbb{S}^{L-1}$ or
$SO(m)$. Let $\phi\in \dot{H}^1(\mathbb R^n, N)$
and $\psi\in (L^\infty\cap {H}^1)(\mathbb R^n, \phi^*TN)$. For $\eps>0$, let
$v_\eps:\R^n\times (0,\infty)\to N$
be a minimizer of the energy
${E}_\eps$ over the admissible space $\mathcal{H}$ together with boundary data $(\phi,\psi)$. Then there exists a subsequence $\eps_i\to 0$ such that $v_i=v_{\eps_i}$ weakly converges to 
a map $v:\mathbb{R}^n\times \R_+\to 
N$ in $H^1(\R^n\times \mathbb{R}_+)$ so that
$$
(\partial_tv, \nabla v)\in L^\infty(\mathbb{R}_+, L^2(\mathbb R^n)).
$$
And $v$ is a weak solution of the wave map equation:
\begin{align}
\partial_t^2v-\Delta v \perp T_v N \ 
\ {\rm{in}}\ \ \mathbb{R}^n\times (0,\infty),   
\end{align}
along with the initial condition
\begin{align}
v(x,0)=\phi(x),\  \partial_tv(x,0)=\psi(x),
\ \ \forall x\in\mathbb{R}^n.
\end{align}
Furthermore, $v$ satisfies the following energy inequality:
\begin{align}
E(v(t))=\int_{\mathbb{R}^n}
(|\partial_tv|^2+|\nabla v|^2)(x,t)\,dx
\le E(v(0))=\int_{\mathbb{R}^n}
(|\psi|^2+|\nabla \phi|^2)(x)\,dx.
\end{align}
\end{theorem}

We would like to point out that, due to the variational approach
of specific forms, the condition on $\psi$ is stronger than the standard assumption $\psi\in L^2(\R^n)$ for wave maps.
We will present the proof in detail for $N=\mathbb{S}^{L-1}$ in Section 2, and point out the necessary changes needed for
the proof of $N=SO(m)$ in Section 3. 

\bigskip

\section{Proof of Theorem \ref{main} for $N=\mathbb{S}^{L-1}$}

In this section, we will establish uniform energy estimates
of $v_\eps$ and then provide a proof of Theorem \ref{main} for
$N=\mathbb{S}^{L-1}.$

First, by defining $u(x,t)=v(x, \eps t)$, one can check that
$E_\eps(v)=\eps^{-3}I_\eps(u)$, where
$I_\eps$ is defined by
\begin{equation}
I_\eps(u)=\iint_{\mathbb R^n\times\R_+} e^{-t}
\Big\{|\partial_t^2 u|^2+{\eps^2}|\nabla u|^2\Big\}\,dxdt,
\end{equation}
and the boundary condition for $u$ becomes
\begin{equation}\label{bdrycond}
u(x,0)=\phi(x), 
\ \ \partial_tu(x,0)=\eps \psi(x), \ x\in\mathbb{R}^n.
\end{equation}
\subsection{Existence of minimizer for $I_\eps$} 
Define the function space
\begin{align*}
\mathcal{H}&=
\Big\{u:\ \mathbb R^n\times \R_+\to \mathbb{S}^{L-1}: (\partial_t^2u,\nabla u)\in L^2_{\rm{loc}}(\mathbb R^n\times \R_+)\ \ {\rm{and}} \\
&\qquad\big[u\big]_{\mathcal{H}}^2
=\iint_{\mathbb{R}^n\times\R_+}
e^{-t}\big(|\partial_t^2u|^2+|\nabla u|^2)(x,t)\,dxdt<\infty
\Big\}.
\end{align*}

Since $|u|=1$ in $\R^n\times (0,\infty)$ for any
$u\in \mathcal{H}$, it is not difficult to verify that for any 
$0<T<\infty$ and compact set $K\subset\mathbb{R}^n$,
$$
\mathcal{H}\hookrightarrow H^2([0,T], L^2(K)),
$$
so that 
$$
(u, \partial_tu)\in C([0, T], L^2(K)), \ \forall u\in \mathcal{H}.
$$
Hence for each $u\in \mathcal{H}$, the traces $u(\cdot, 0)$
and $\partial_t u(\cdot, 0)$ are  well defined as elements of $L^2(K)$, see also \cite[Remark 2.1]{Serra-Tilli2012}.

\smallskip
Given $\phi\in \dot{H}^1(\mathbb R^n,\mathbb S^{L-1})$ 
and $\psi\in ({H}^1\cap L^\infty)(\mathbb R^n,\mathbb \phi^*T\mathbb{S}^{L-1})$, we define
\[
\mathcal{H}_{\phi,\eps\psi}
=\Big\{u\in \mathcal{H}: \ u(x,0)=\phi(x),
\ \partial_t u(x,0)=\eps\psi(x)\ {\rm{for}}\ x\in \R^n\Big\}.
\]
It follows from Lemma \ref{nonempty} that
$\mathcal{H}_{\phi,\eps\psi}\not=\emptyset$.
Now we can show the 
existence of minimizer of $I_\eps$ over $\mathcal{H}_{\phi,\eps\psi}$.
\begin{lemma} \label{existenemin} Assume $\phi\in \dot{H}^1(\mathbb R^n,\mathbb S^{L-1})$ 
and $\psi\in ({H}^1\cap L^\infty)(\mathbb R^n,\mathbb \phi^*T\mathbb{S}^{L-1})$. Then there exists a map $u_\eps\in\mathcal{H}_{\phi,\eps\psi}$ such that
\begin{align}\label{minimizer}
I_\eps(u_\eps)\le I_\eps(v), \ \forall v\in\mathcal{H}_{\phi,\eps\psi}. \end{align}
\end{lemma}
\begin{proof} First, it follows from Lemma \ref{nonempty} below that
\[
\alpha:=\inf\Big\{I_\eps(v): 
\ v\in\mathcal{H}_{\phi,\eps\psi}\Big\}
\le \eps^2\int_{\mathbb R^n}|\nabla\phi|^2+C\eps^3<\infty,
\]
where $C>0$ depends on 
$\|\nabla\phi\|_{L^2(\mathbb{R}^n)}$,
$\|\nabla\psi\|_{L^2(\mathbb{R}^n)}$,
and $\|\psi\|_{L^4(\mathbb{R}^n)}$.

Now let $\{u_i\}\subset\mathcal{H}_{\phi,\eps\psi}$ be a minimizing sequence of $\alpha$, i.e.,
$$
I_\eps(u_i)\rightarrow \alpha.
$$
By the weak compactness of $\mathcal{H}_{\phi,\eps\psi}$, there exists a map $u_\eps:\R^n\times\mathbb R_+\to\mathbb R^L$,
with $\big[u\big]_{\mathcal{H}}<\infty$,
such that, after passing to a subsequence if necessary,
$$
(\partial_t^2u_i,\nabla u_i)\rightharpoonup
(\partial_t^2u_\eps,\nabla u_\eps) 
\ \ {\rm{in}}\ \ L^2(\mathbb R^n),
$$
and
$$
(u_i,\partial_tu_i)\rightarrow (u_\eps, \partial_tu_\eps)
\ \ {\rm{in}}\ \ L^2_{\rm{loc}}(\mathbb{R}^n\times \R_+)\ \ {\rm{and\ a.e.\ in }}
\ \ \R^n\times\mathbb R_+.
$$
In particular, $|u_\eps|=1$ a.e. in 
\ $\R^n\times\mathbb R_+$, and $(u_\eps,\partial_t u_\eps)\big|_{t=0}=(\phi,\eps\psi)$, that is, $u_\eps\in \mathcal{H}_{\phi,\eps\psi}$. 
Hence 
\[
I_\eps(u_\eps)\ge\alpha.
\]
On the other hand, by the lower semi continuity of $I_\eps$
on $\mathcal{H}$, we have
\[
I_\eps(u_\eps)\le \liminf_{i}I_\eps(u_i)=\alpha.
\]
Thus $I_\eps(u_\eps)=\alpha$. This
completes the proof.
\end{proof}

\subsection{Uniform Energy Estimates}
\begin{lemma} \label{nonempty} Assume $\phi\in \dot{H}^1(\mathbb R^n,\mathbb S^{L-1})$ 
and $\psi\in ({H}^1\cap L^\infty)(\mathbb R^n,\mathbb \phi^*T\mathbb{S}^{L-1})$.
Then 
\begin{align}\label{energy_ineq1}
\alpha:=\inf\Big\{I_\eps(u):\  u\in \mathcal{H}_{\phi,\eps\psi}\Big\}
\le \eps^2\int_{\mathbb{R}^n}
|\nabla\phi|^2+C\eps^3.  
\end{align}
\end{lemma}
\begin{proof} Since 
$\phi(x)+t\eps \psi(x)$ satisfies
$$
|\phi(x)+t\eps \psi(x)|^2=
|\phi(x)|^2+t^2\eps^2|\psi(x)|^2
\ge 1,
$$
we can define the $\mathbb{S}^{L-1}$-valued map by
$$w_\eps(x)=\displaystyle\frac{\phi(x)+t\eps \psi(x)}
{|\phi(x)+t\eps \psi(x)|}, \ (x,t)\in\R^n\times\mathbb R_+.
$$
By the definition of $\alpha$, we have
$$
\alpha\le I_\eps(w_\eps).
$$
To estimate $I_\eps(w_\eps)$,
we first calculate
\begin{align*}
\nabla w_\eps(\cdot,t)&=
\frac{(\nabla\phi+t\eps\nabla\psi)
|\phi+t\eps\psi|^2-(\phi+t\eps\psi)\langle\phi+t\eps\psi,\nabla\phi+t\eps\nabla\psi\rangle}
{|\phi+t\eps\psi|^3}\\
&=\frac{1}{|\phi+t\eps\psi|}\Big(
(\nabla\phi+t\eps\nabla\psi)-
\langle \nabla\phi+t\eps\nabla\psi, \frac{\phi+t\eps\psi}{|\phi+t\eps\psi|}\rangle\frac{\phi+t\eps\psi}{|\phi+t\eps\psi|}\Big)
\end{align*}
so that
$$
|\nabla w_\eps|(x,t)
\le |\nabla\phi+t\eps\nabla\psi|
\le |\nabla\phi|+t\eps|\nabla\psi|.
$$
Also we calculate
$$
\partial_tw_\eps(\cdot,t)
=\frac{\eps\psi}{|\phi+t\eps\psi|}
-\frac{\eps (\phi+t\eps\psi)\langle \phi+t\eps\psi,\psi\rangle}{|\phi+t\eps\psi|^3}
$$
so that
\begin{align*}
 \partial_t^2w_\eps(\cdot, t)
 &=-\frac{2\eps^2\psi\langle \phi+t\eps\psi,\psi\rangle}{|\phi+t\eps\psi|^3}
 -\frac{\eps^2|\psi|^2(\phi+t\eps\psi)}
 {|\phi+t\eps\psi|^3}\\ &\ \ \ +3\frac{\eps^2(\phi+t\eps\psi)\langle \phi+t\eps\psi,\psi\rangle^2}
 {|\phi+t\eps\psi|^5}.
\end{align*}
Hence
\begin{align*}
|\partial_t^2w_\eps|(\cdot, t)
&\le 6\eps^2|\psi|^2
\end{align*}
This implies that 
\begin{align*}
I_\eps(w_\eps)
 &\le  \iint_{\mathbb{R}^n\times\R_+} e^{-t}(|\partial_t^2w_\eps|^2+\eps^2|\nabla w_\eps|^2)\,dxdt\\
 &\le \int_{\mathbb{R}^n}
(\eps^2|\nabla\phi|^2+C\eps^3|\nabla\phi||\nabla\psi|+C\eps^4 (|\nabla\psi|^2+|\psi|^4))\,dxdt\\
&\le\eps^2\int_{\mathbb{R}^n}|\nabla\phi|^2\,dx+O(\eps^3).
\end{align*}
Here we have used the fact that
\[
\|\psi\|_{L^4(\R^n)}\le \|\psi\|_{L^2(\R^n)}^\frac12
\|\psi\|_{L^\infty(\R^n)}^\frac12.
\]
This completes the proof. \end{proof} 

To proceed, we introduce a few notations now. 
For any minimizer $u_\eps$ of $I_\eps$ over $\mathcal{H}_{\phi,\eps\psi}$, define
$$
L_\epsilon(t)=\int_{\mathbb{R}^n}
(|\partial_t^2u_\eps|^2+\eps^2|\nabla u_\eps|^2)(x,t)\,dx,
$$
$$
H_\eps(t)=\int_t^\infty e^{-s}L_\eps(s)\,ds,
$$
$$
I_\epsilon(t)=\frac12 \int_{\mathbb{R}^n} |\partial_tu_\eps|^2(x,t)\,dx,
$$
$$
D_\eps(t)=\int_{\mathbb{R}^n} |\partial_t^2u_\eps|^2(x,t)\,dx,
$$
and
$$
E_\eps(t)=I_\eps(t)-I_\eps'(t) 
+\frac12 e^t H_\eps(t).
$$

\begin{lemma} \label{energyestimate}  Assume $\phi\in \dot{H}^1(\mathbb R^n,\mathbb S^{L-1})$ 
and $\psi\in ({H}^1\cap L^\infty)(\mathbb R^n,\mathbb \phi^*T\mathbb{S}^{L-1})$.
Let $u_\eps$ be a minimizer of $I_\eps$ in $\mathcal{H}_{\phi,\eps\psi}$. Then
$E_\eps(t)$ satisfies
\begin{equation}\label{energy_identity}
 E_\epsilon'(t)=-2D_\eps(t).    
\end{equation}
In particular, $E_\epsilon\in W^{1,1}_{\rm{loc}}(\mathbb{R}_+)$, and is monotonically non-increasing with respect to $t$.    
\end{lemma}
\begin{proof} The idea of proof is based on the inner variation of $u_\eps$ along the $t$-direction.
For the detail, we follow closely the argument by \cite[Proposition 3.1]{Serra-Tilli2012}.
For any $\eta\in C_0^\infty(\mathbb{R}_+)$, define
$$
g(t)=\int_0^t\eta(s)\,ds.
$$
Then $g\in C^\infty(\mathbb{R}_+)$ vanishes near $t=0$. 

Write $u$ for $u_\eps$. For a sufficiently small $\delta\in\mathbb{R}$, define
$$\alpha_\delta(t)=t-\delta g(t) 
\ \ {\rm{and}}\ \ U_\delta(x,t )=u(x, \alpha_\delta(t)), 
\ (x,t)\in\R^n\times\mathbb{R}_+.$$
Then $U_\delta$ satisfies the boundary condition (2.2), and $u\equiv U_0$ so that by the minimality, we have 
$$
I_\eps(u)\le I_\eps(U_\delta), 
$$
and hence
$$
\frac{d}{d\delta}\big|_{\delta=0}I_\eps
(U_\delta)=0.
$$
Since $\alpha_\delta$ is invertible when
$\delta$ is sufficiently small, we  denote $\beta_\delta$ as the inverse so that $\beta_\delta\circ\alpha_\delta(t)=t$.
By change of variables, we can rewrite
$I_\eps(U_\delta)$ as
\begin{align*}
I_\eps(U_\delta)
&=\iint_{\mathbb{R}^n\times\R_+}
(e^{-\beta_\delta(s)}\beta_\delta'(s))
\big(|\partial_t^2u(\alpha_\delta')^2+\partial_tu\alpha_\delta''|^2+\eps^2 |\nabla u|^2(x,s)\big)\,dxds.
\end{align*}
Direct calculations yield
\[
\frac{d}{d\delta}\big|_{\delta=0}\big(e^{-\beta_\delta}\beta_\delta'\big)
=e^{-t}(g'(t)-g(t)),
\]
and
\[
\frac{d}{d\delta}\big|_{\delta=0}\alpha'_\delta(t)=g'(t),
\ \ \frac{d}{d\delta}\big|_{\delta=0}\alpha''_\delta(t)=g''(t).
\]
Using these identities, we can deduce
\begin{align*}
0&=\frac{d}{d\delta}\big|_{\delta=0}I_\eps
(U_\delta)\\
&=\iint_{\mathbb{R}^n\times\R_+}
e^{-t}(g'-g)(|\partial_t^2u|^2+\eps^2|\nabla u|^2)(t,x)\,dxdt\\
&\ \ -\iint_{\mathbb{R}^n\times\R_+}
e^{-t}(4|\partial_t^2u|^2g'+2\langle \partial_t^2u, \partial_t u\rangle g'')(t,x)\,dxdt\\
&=\int_0^\infty
e^{-t}(g'(t)-g(t))L_\eps(t)\,dt-\int_0^\infty
e^{-t}(4D_\eps(t) g'(t)+2 I_\eps'(t) g''(t))\,dt.
\end{align*}
Observe that
\begin{align*}
-\int_0^\infty e^{-t} g(t)L_\epsilon(t)\,dt 
&=\int_0^\infty g(t)H_\eps'(t)\,dt\\
&=g(t)H_\eps(t)\big|_{0}^\infty-\int_0^\infty H_\eps(t)g'(t)\,dt\\
&=-\int_0^\infty H_\eps(t)g'(t)\,dt,
\end{align*}
since $g(0)=H_\eps(\infty)=0$.
Observe also that
$$g'(t)=\eta(t)
\ \ {\rm{and}}\ \ g''(t)=\eta'(t).$$
Therefore, we obtain that
\begin{align}
 \int_0^\infty
\big(L_\eps(t)-e^t H_\eps(t)-4D_\eps(t)\big) e^{-t}\eta(t)\,dt-2\int_0^\infty
I_\eps'(t) e^{-t}\eta'(t))\,dt=0.
\end{align}
By integration by parts, we have
\[
\int_0^\infty I_\eps'(t) e^{-t}\eta'(t)\,dt
=\int_0^\infty I_\eps'(t) e^{-t}\eta(t)\,dt
-\int_0^\infty I_\eps''(t) e^{-t}\eta(t)\,dt.
\]
Thus we arrive at
\[
 \int_0^\infty
\Big[\big(L_\eps(t)-e^t H_\eps(t)-4D_\eps(t)-2I_\eps'(t)\big)
+2I_\eps''(t)\Big]e^{-t}\eta(t)\,dt=0.
\]
This, combined with
$(e^t H_\epsilon(t))'
=e^t H_\eps(t)-L_\eps(t)$,
yields
\begin{equation}\label{energy_ineq2}
E_\eps'(t)=\Big(\frac12 e^t H_\epsilon(t)+I_\eps(t)-I_\eps'(t)\Big)'=-2D_\eps(t)\le 0.
\end{equation}
In particular, $E_\eps(t)$ is monotonically decreasing with respect to $t$. This completes the proof.
\end{proof}

\medskip
Since
\[
-(e^{-t}I_\eps(t))'+\frac12 H_\eps(t)
=e^{-t}E_\eps(t),
\]
it follows from \eqref{energy_ineq2} that we can integrate from $t$ to $\infty$ to get
\[
e^{-t}I_\eps(t)
+\frac12\int_t^\infty H_\eps(\tau)\,d\tau 
=\int_t^\infty e^{-\tau} E_\eps(\tau)\,d\tau
\le E_\eps(t) \int_t^\infty e^{-\tau}\,d\tau
=e^{-t}E_\eps(t),
\]
this yields
\begin{equation}\label{non-negativity}
I_\eps(t)+\frac{e^t}2\int_t^\infty H_\eps(\tau)\,d\tau\le E_\eps(t).
\end{equation}

From the global energy inequality \eqref{energy_ineq1}, we have that
$$
H_\eps(0)=I_\eps(u_\eps)\le C\eps^2
$$
Since $H_\eps(t)\le H_\eps(0)$ for all
$t\ge 0$, we then have
\begin{equation}\label{E1}
H_\eps(t)\le C\eps^2, \ \forall t\ge 0.
\end{equation}
Since
$$
\int_0^\infty e^{-t}D_\eps(t)\,dt\le I_\eps(u_\eps)\le C\eps^2,
$$
we see that
\begin{equation}\label{E2}
\int_0^1 D_\eps(t)\,dt\le e\int_0^\infty e^{-t}D_\eps(t)\,dt\le C\eps^2.
\end{equation}
This, combined with $\partial_tu_\eps(x,0)=\eps\psi(x)$, implies
\begin{equation}\label{E3}
\int_0^1 I_\eps(t)
\le C\big(\int_0^1 D_\eps(t)\,dt+\int_{\mathbb{R}^m} |\partial_tu_\eps(x,0)|^2\,dx\big)
\le C\eps^2.
\end{equation}
Hence, by H\"older inequality, we have
\begin{equation}\label{E4}
\int_0^1 |I_\eps'(t)|\,dt
\le \big(\int_0^1 D_\eps(t)\,dt\big)^\frac12\big(\int_0^1 I_\eps(t)\,dt\big)^\frac12\le C\eps^2.
\end{equation}

Putting \eqref{E1}, 
\eqref{E3}, and \eqref{E4} together, one has
\begin{equation}\label{energy_ineq3}
\int_0^1 E_\eps(t)\,dt
=\int_0^1 (I_\eps(t)-I_\eps'(t)+\frac12 e^t H_\eps(t))\,dt\le C\eps^2.
\end{equation}

Observe that from \eqref{energy_identity}, we have
that
\begin{equation}
E_\eps(0)=E_\eps(t)+2\int_0^t D_\eps(\tau)\,d\tau,     
\end{equation}
so that by integrating $t\in [0,1]$, we obtain that
\begin{align*}
E_\eps(0)&=\int_0^1 E_\eps(t)\,dt+2\int_0^1\int_0^t D_\eps(\tau)\,d\tau dt\\
&\le \int_0^1 E_\eps(t)\,dt+2\int_0^1D_\eps(t)\,dt\\
&\le C\eps^2,
\end{align*}
where we have applied \eqref{energy_ineq3} and \eqref{E2} in the last step. In particular, we also have  that
\begin{align}\label{energy_ineq4}
0\le E_\eps(t)\le E_\eps(0)\le C\eps^2, 
\end{align}
and
\begin{align}\label{energy_ineq5}
\iint_{\mathbb{R}^n\times\R_+} |\partial_t^2u_\eps(x,t)|^2\,dxdt
=-\frac12\int_0^\infty E_\eps'(t)\,dt
&=\frac12(E_\eps(0)-E_\eps(\infty))\nonumber\\
&=\frac12 E_\eps(0)\le C\eps^2.
\end{align}

Observe that by the monotonic decreasing property of $H_\eps(t)$ and
\eqref{non-negativity}, we have
\begin{align*}
 e^t H_\eps(t+1)\le e^t\int_t^{t+1}H_\eps(\tau)\,d\tau 
 \le e^t\int_t^\infty H_\eps(\tau)\,d\tau
 \le 2E_\eps(t)\le C\eps^2,\ \forall t>0.
\end{align*}
This implies
\begin{align}\label{energy_ineq6}
\int_t^{t+1}\int_{\mathbb{R}^n}
|\nabla u_\eps(x,\tau)|^2\,dxd\tau
&\le
C\frac{1}{\eps^2}
e^{t+1}\int_t^{t+1}e^{-s}L_\eps(s)\,ds\nonumber\\
&\le C\eps^{-2}e^{t+1}H_\eps(t) \le C, 
\ \forall t\ge 0.
\end{align}

\medskip
\subsection{Dual estimates of $(\partial_t^2u_\eps\times u_\eps)(\cdot,t)$}

In this subsection, we prove the following.

\begin{lemma} \label{dualestimate} Assume $\phi\in \dot{H}^1(\mathbb R^n,\mathbb S^{L-1})$ 
and $\psi\in ({H}^1\cap L^\infty)(\mathbb R^n,\mathbb \phi^*T\mathbb{S}^{L-1})$.
Let $u_\eps$ be a minimizer of $I_\eps$ in $\mathcal{H}_{\phi,\eps\psi}$. There exists  $C>0$ such that for any $h\in L^\infty(\mathbb{R}^n)\cap \dot{H}^1(\mathbb{R}^n)$, if
$u_\eps$ is a minimizer of $I_\eps$, then for a.e. $t>0$, 
\begin{align}\label{dual_est}
 \Big|\int_{\mathbb{R}^n} 
 \partial_t^2u_\eps(x,t)\times u_\eps(x,t)\cdot h(x)\,dx\Big|\le Ce^{t}\eps^2\Big(\|h\|_{L^\infty(\mathbb {R}^n)}+\|\nabla h\|_{L^2(\mathbb{R}^n)}\Big).
\end{align}    
\end{lemma}
\begin{proof} Since $u_\eps$ is a minimizer of $I_\eps$ over
$\mathcal{H}_{\phi,\psi}^\eps$, it is a weak solution of the Euler-Lagrange equation:
\begin{align}\label{EL1}
\Big[\partial_t^2(e^{-t}\partial_t^2u_\eps)-\nabla\cdot(\eps^2 e^{-t}\nabla u_\eps)\Big]\times u_\eps=0
\ \ {\rm{in}}\ \ \R^n\times(0,+\infty). 
\end{align}
Assume $h\in C_0^\infty(\mathbb{R}^n)$. As in \cite[Theorem 4.2]{Serra-Tilli2012}, let $\eta\in C^{1,1}(\mathbb{R})$ be given by
\[
\eta(t)=\begin{cases} 0 & t\le 0,\\
t^2 & 0<t<1,\\
2t-1 & t\ge 1.  
\end{cases}
\]
For $T>0$ and $\delta>0$, define
$$
\eta_\delta(t)=\delta\eta(\frac{t-T}{\delta}).
$$
Testing \eqref{EL1} by $\eta_\delta(t)h(x)$, we have that
\begin{align*}
0&=\iint_{\mathbb{R}^n\times\R_+}
\Big[e^{-t}\partial_t^2u_\eps\times \partial_t^2(u_\eps \eta_\delta h)+\eps^2 e^{-t}\nabla u_\eps\times \nabla(u_\eps\eta_\delta h)\Big]\,dxdt\\
&=\iint_{\mathbb{R}^n\times\R_+}
e^{-t}[\partial_t^2u_\eps\times (\partial_t^2u_\eps\eta_\delta +u_\eps \eta_\delta''+2\partial_t u_\eps\eta_\delta')h+\eps^2 
\nabla u_\eps\times (\nabla u_\eps h+u_\eps \nabla h) \eta_\delta]\,dxdt\\
&=\iint_{\mathbb{R}^n\times\R_+}
e^{-t}[\partial_t^2u_\eps\times (u_\eps \eta_\delta''+2\partial_tu_\eps\eta_\delta')h+\eps^2\nabla u_\eps\times (u_\eps \nabla h) \eta_\delta]\,dxdt.
\end{align*}
Thus, we have
\begin{align}\label{EL2}
\iint_{\mathbb{R}^n\times\R_+}
e^{-t}\partial_t^2u_\eps\times u_\eps \eta_\delta''h\,dxdt&+2\iint_{\mathbb{R}^n\times\R_+}e^{-t}\partial_t^2u_\eps\times \partial_tu_\eps\eta_\delta' h\,dxdt\nonumber\\
&+\eps^2\iint_{\mathbb{R}^n\times\R_+}e^{-t}\nabla u_\eps\times u_\eps \nabla h \eta_\delta\,dxdt=0.
\end{align}
Since $\eta_\delta''(t)=2\delta^{-1}\chi_{(T,T+\delta)}(t)$, the first term of \eqref{EL2} becomes
\begin{align*}
\iint_{\mathbb{R}^n\times\R_+}
e^{-t}\partial_t^2u_\eps\times u_\eps \eta_\delta''h\,dxdt
&=\frac{2}{\delta}\int_{T}^{T+\delta} 
\int_{\mathbb{R}^n} e^{-t} \partial_t^2 u_\eps\times u_\eps h\,dxdt
\nonumber\\
&\rightarrow 2e^{-T}
\int_{\mathbb{R}^n} \partial_t^2u_\eps\times u_\eps h\,dx,
\ {\rm{as}}\ \ \delta\to 0.
\end{align*}
The second term of \eqref{EL2} can be estimated as follows.
\begin{align*}
&2\iint_{\mathbb{R}^n\times\R_+}e^{-t}
\partial_t^2u_\eps\times \partial_tu_\eps\eta_\delta' h\,dxdt\\
&
=4\int_{T}^{T+\delta} \frac{t-T}{\delta}
e^{-t}\int_{\mathbb{R}^n} \partial_t^2u_\eps\times 
\partial_tu_\eps h\,dxdt\\
&\quad+4\int_{T+\delta}^\infty  
e^{-t}\int_{\mathbb{R}^n} \partial_t^2u_\eps\times 
\partial_tu_\eps h\,dxdt,
\end{align*}  
hence we have
\begin{align*}
&\Big|2\iint_{\mathbb{R}^n\times\R_+}e^{-t}\partial_t^2u_\eps\times \partial_tu_\eps\eta_\delta' h\,dxdt\Big|\\
&\le 4\int_{T}^{\infty} e^{-t}\int_{\mathbb{R}^n} 
|\partial_t^2u_\eps\times \partial_tu_\eps| |h|\,dxdt\\
&\le C\|h\|_{L^\infty(\mathbb{R}^n)}
\Big(\iint_{\mathbb{R}^n\times\R_+}
e^{-t}|\partial_t^2u_\eps|^2\,dxdt\Big)^\frac12 
\Big(\iint_{\mathbb{R}^n\times\R_+}
e^{-t}|\partial_tu_\eps|^2\,dxdt\Big)^\frac12\\
&\le 
C\|h\|_{L^\infty(\mathbb{R}^n)}
\Big(\iint_{\mathbb{R}^n\times\R_+}
e^{-t}|\partial_t^2u_\eps|^2\,dxdt\Big)^\frac12\\
&\quad\cdot\Big(\iint_{\mathbb{R}^n\times\R_+}
e^{-t}|\partial_t^2u_\eps|^2\,dxdt+\int_{\mathbb{R}^n}|\partial_tu_\eps(x,0)|^2\,dx\Big)^\frac12\\
&\le C\eps^2 \|h\|_{L^\infty(\mathbb{R}^n)}.
\end{align*}  

The third term of \eqref{EL2} can be
estimated by 
\begin{align*}
&\lim_{\delta\to 0} 
\Big|\iint_{\mathbb{R}^n\times \R_+}e^{-t}\nabla u_\eps\times u_\eps \nabla h \eta_\delta\,dxdt\Big|\\
&=2\Big|\iint_{\mathbb{R}^n\times\R_+}e^{-t}(t-T)^+\nabla u_\eps\times u_\eps \nabla h\,dxdt\Big|\\
&\le 2\iint_{\mathbb{R}^n\times\R_+}e^{-t}(t-T)^+|\nabla u_\eps||\nabla h|\,dxdt\\
&\le 2\sum_{k=0}^\infty
e^{-k}(k+1)\Big(\int_{T+k}^{T+k+1}\int_{\mathbb{R}^n}|\nabla u_\eps|^2\,dxdt\Big)^\frac12 
\Big(\int_{T+k}^{T+k+1}\int_{\mathbb{R}^n}|\nabla h|^2\,dxdt\Big)^\frac12\\
&\le C\|\nabla h\|_{L^2(\mathbb{R}^n)}
\sup_{k\ge 0} \Big(\int_{T+k}^{T+k+1}\int_{\mathbb{R}^n}|\nabla u_\eps|^2\,dxdt\Big)^\frac12\\
&\le C\|\nabla h\|_{L^2(\mathbb{R}^n)}.
\end{align*}
Substituting these estimates into \eqref{EL2} yields \eqref{dual_est}
holds for a.e. $T>0$.
\end{proof}

\subsection{Passage to the limit}
In this subsection, we illustrate how to pass to the limit $v$ of $v_\eps(x,t)=u_\eps(x, t/\eps)$ so that
$v$ is a weak solution of the wave map equation which satisfies the weak-type
global energy inequality. 

First, let us collect all the estimates for $v_\eps$ after translating those for $u_\eps$. They are as follows.

\begin{align}\label{H1-est}
\int_0^T \int_{\mathbb{R}^n}
|\nabla v_\eps|^2(x,t)\,dxdt\le CT.
\end{align}
\begin{align}\label{t-derivate-est}
\int_{\mathbb{R}^n}
|\partial_tv_\eps|^2(x,t)\,dx\le C, \ 
{\rm{a.e.}}\ t>0,
\end{align}
\begin{align}\label{tt-derivative-est}
\iint_{\mathbb{R}^n\times\R_+}
|\partial_t^2v_\eps|^2(x,t)\,dxdt\le C\eps^{-1},
\end{align}
and
\begin{align}\label{dual-est3}
&\Big|\int_{\mathbb{R}^n}\partial_t^2v_\eps(x,t)\times v_\eps(x,t) h(x)\,dx\Big|\nonumber\\
&\le C(T)\big(\|h\|_{L^\infty(\mathbb{R}^n)}
+\|\nabla h\|_{L^2(\mathbb{R}^n)}\big),
\ {\rm{a.e.}}\ 0<t<T.
\end{align}

After passing to a subsequence, we may assume that there exists a map $v\in \dot{H}^1(\mathbb{R}^n\times\mathbb{R}_+,\mathbb{S}^{L-1})$ such that
\[
(\partial_tv_\eps, \nabla v_\eps)
\rightharpoonup (\partial_tv, \nabla v) 
\ \ {\rm{in}}\ \ L^2(\mathbb{R}^n\times\mathbb{R}_+).
\]
Since $v_\eps$ is a weak solution of
\[
\Big[\partial_t^2\big(e^{-t/\eps} \eps^2 \partial_t^2v_\eps)
-\nabla\cdot(e^{-t/\eps}\nabla v_\eps)
\Big]\times v_\eps=0 \ \ {\rm{in}}\ \ \mathbb{R}^n\times\mathbb{R}_+,
\]
we have, for any test function $\eta\in C_0^\infty(\mathbb{R}^n\times\mathbb{R}_+)$,
we have 
\[
\iint_{\mathbb{R}^n\times\R_+}\big(e^{-t/\eps} \eps^2 \partial_t^2v_\eps\big)\times \partial_t^2(v_\eps\eta)+(e^{-t/\eps}\nabla v_\eps)\times \nabla(v_\eps\eta)\,dxdt
=0. 
\]
After expansion and cancellations of certain terms, this identity reduces to
\[
\iint_{\mathbb{R}^n\times\R_+}\big(e^{-t/\eps} \eps^2 \partial_t^2v_\eps\big)\times (2\partial_tv_\eps\eta'+v_\eps \eta'')+(e^{-t/\eps}\nabla v_\eps)\times (v_\eps\nabla\eta)\,dxdt
=0. 
\]
Let $\hat{\eta}=e^{-t/\eps}\eta$. Then we have
$$
2\partial_tv_\eps\eta'+\partial_t^2v_\eps\eta
=e^{t/\eps}\Big[2\partial_tv_\eps\big(\hat{\eta}'+\frac{1}{\eps}\hat{\eta}\big)
+v_\eps\big(\hat{\eta}''+\frac{2}{\eps}\hat{\eta}'+\frac{1}{\eps^2}\hat{\eta}\big)\Big],
\ \ \nabla\eta=e^{t/\eps}\nabla\hat{\eta}.
$$
Substituting this into the equation above, we obtain that
\begin{align*}
&\iint_{\mathbb{R}^n\times\R_+}
\big[-\partial_tv_\eps\times \partial_t\big(v_\eps\hat{\eta}\big)
+\nabla v_\eps\times (v_\eps\nabla\hat{\eta})\big]\,dxdt\\
&=-\iint_{\mathbb{R}^n\times\R_+}
\partial_t^2v_\eps\times \Big[2\partial_tv_\eps\big(\eps^2\hat{\eta}'+\eps
\hat{\eta}\big)+v_\eps\big(\eps^2\hat{\eta}''+2\eps\hat{\eta}'\big)\Big]\,dxdt.
\end{align*}
From \eqref{t-derivate-est} and
\eqref{tt-derivative-est}, it is easy to see that
\[
\iint_{\mathbb{R}^n\times\R_+}
\partial_t^2v_\eps\times \Big[2\partial_tv_\eps\big(\eps^2\hat{\eta}'+\eps
\hat{\eta}\big)+v_\eps\big(\eps^2\hat{\eta}''+2\eps\hat{\eta}'\big)\Big]\,dxdt\to 0
\]
as $\eps\to 0$. Thus, after sending $\eps\to 0$ in the equation above, we arrive at
\[
\iint_{\mathbb{R}^n\times\R_+}
\big[-\partial_tv\times \partial_t\big(v\hat{\eta}\big)
+\nabla v\times (v\nabla\hat{\eta})\big]\,dxdt=0,
\]
this is,
$v$ is a weak solution of
\[
\partial_t^2v\times v-\nabla\cdot(\nabla v\times v)=0
\ \ {\rm{in}}\ \ \mathbb{R}^n\times\mathbb{R}_+.
\]
This is equivalent to that $v$ is a weak solution of the wave map equation:
\[
v_{tt}-\Delta v\perp T_v \mathbb{S}^{L-1}
\ \ {\rm{in}}\ \ \mathbb{R}^n\times\mathbb{R}_+.
\]

Next, we want to verify that $v$ satisfies the initial condition
\[
v(x,0)=\phi(x),\ \ \partial_tv(x,0)=\psi(x) 
\ \ {\rm{in}}\ \ \mathbb{R}^n.
\]
The first condition follows from the compactness of traces. For any $T>0$ and $R>0$, 
$$
H^1(B_R\times [0,T])\hookrightarrow L^2(B_R\times\{0\}),
$$
and $v_\eps(x,0)=\phi(x)$ in $\mathbb{R}^n$.

The second condition follows from \eqref{dual-est3} and the following argument, similar to \cite[page 1571]{Serra-Tilli2012}. For $\eta(t)\in C_0^\infty(\mathbb{R})$, we have
\begin{align*}
\int_0^\infty\eta(t)\int_{\mathbb{R}^n}
\partial_t(\partial_tv_\eps\times v_\eps) h\,dxdt
&=-\int_0^\infty\eta'(t)\int_{\mathbb{R}^n}
\partial_tv_\eps\times v_\eps h\,dxdt\\
&\quad-\eta(0)
\int_{\mathbb{R}^n}\psi(x)\times \phi(x) h(x)\,dx.
\end{align*}
Passing to the limit $\eps\to 0$ in this identity and applying \eqref{dual-est3}, we obtain
\begin{align*}
\int_0^\infty\eta(t)\int_{\mathbb{R}^m}
\partial_t(\partial_tv\times v) h\,dxdt
&=-\int_0^\infty\eta'(t)\int_{\mathbb{R}^n}
\partial_tv\times v h\,dxdt\\
&\quad-\eta(0)
\int_{\mathbb{R}^n}\psi(x)\times \phi(x) h(x)\,dx.
\end{align*}
On the other hand, by integration by parts we have
\begin{align*}
\int_0^\infty\eta(t)\int_{\mathbb{R}^n}
\partial_t(\partial_tv\times v)h\,dxdt
&=-\int_0^\infty\eta'(t)\int_{\mathbb{R}^n}
\partial_tv\times v h\,dxdt\\
&\quad-\eta(0)\int_{\mathbb{R}^n}
\partial_tv(x,0)\times \phi (x) h(x)\,dx.
\end{align*}
Hence
\[
\eta(0)
\int_{\mathbb{R}^n}\psi(x)\times \phi(x) h(x)\,dx
=\eta(0)\int_{\mathbb{R}^n}
\partial_tv(x,0)\times \phi (x) h(x)\,dx.
\]
Since $\eta$ and $h$ are arbitrarily chosen,
we conclude that
\[
(\partial_tv(x,0)-\psi(x))\times \phi(x)=0,
\ x\in\mathbb{R}^n.
\] 
Since $\partial_tv(x,0)-\psi(x)\in T_{\phi(x)}\mathbb{S}^{L-1}$, or equivalently, 
\[
\langle \partial_tv(x,0)-\psi(x),\phi(x)\rangle=0,\ \
x\in\mathbb{R}^m.
\]
Thus we obtain
\[
\partial_tv(x,0)=\psi(x),\ \ x\in\mathbb{R}^n.
\]
This completes the proof, except for part of the global energy inequality. 
\qed

\medskip
\subsection{Global Energy Inequality}

In this subsection, we closely follow the arguments of \cite{Serra-Tilli2012} to show the following result.
\begin{theorem}\label{globalenergyineq1}
For $v$ obtained by the previous subsection, 
\begin{align}\label{globalenergyineq2}
{E}(v(t))
=\int_{\mathbb{R}^n}(|\partial_tv|^2+|\nabla v|^2)(x,t)\,dx
\le {E}(v(0))=\int_{\mathbb{R}^n}(|\psi|^2+|\nabla \phi|^2)(x)\,dx    
\end{align}
holds for a.e. $t>0$.
\end{theorem}
\begin{proof} Applying
Lemma \ref{sharp_est0} below and \eqref{non-negativity} above, we can employ the same arguments as in 
\cite[pages 1572-1573]{Serra-Tilli2012} to establish \eqref{globalenergyineq2}.   
\end{proof}

According to \cite{Serra-Tilli2012}, the key step to establish Theorem \ref{globalenergyineq1} is the following Lemma
that concerns an improved estimate of \eqref{energy_ineq4}. Namely, 

\begin{lemma} \label{sharp_est0} Assume $\phi\in \dot{H}^1(\mathbb R^n,\mathbb S^{L-1})$ 
and $\psi\in ({H}^1\cap L^\infty)(\mathbb R^n,\mathbb \phi^*T\mathbb{S}^{L-1})$.
Let $u_\eps$ be a minimizer of $I_\eps$ in $\mathcal{H}_{\phi,\eps\psi}$.  Then
\begin{align}\label{sharp_est}
 E_\eps(0)
 =I_\eps(0)-I_\eps'(0)+\frac12 H_\eps(0)\le
 \frac12\eps^2\int_{\mathbb{R}^n}
 (|\nabla\phi|^2+|\psi|^2)(x)\,dx+C\eps^3.
\end{align}
    
\end{lemma}
\begin{proof} From the definition, we know
\begin{align}\label{I-est}
I_\eps(0)=\frac12\int_{\mathbb{R}^n}|\partial_tu_\eps(x,0)|^2\,dx
=\frac12\int_{\mathbb{R}^n}\eps^2|\psi(x)|^2\,dx. 
\end{align}
By \eqref{energy_ineq1}, we have 
\begin{align}\label{H-est}
\frac12H_\eps(0)
=\frac12 I_\eps(u_\eps)
\le\frac12\big(\int_{\mathbb{R}^n}
\eps^2|\nabla\phi(x)|^2\,dx+C\eps^3\big).
\end{align}
Now we claim that
\begin{align}\label{I'-est}
 |I_\epsilon'(0)|
 \le C\eps^3.
\end{align}
To do it, first observe that 
$|u_\eps(x,t)|=1$ implies that
$\langle \partial_tu_\eps, u_\eps\rangle (x,t)=0$
and hence
\[
|\partial_tu_\eps|^2(x,t)
=|\partial_tu_\eps\times u_\eps|^2(x,t),
\ (t,x)\in \mathbb{R}^n\times \mathbb{R}_+
\]
so that
\[
I_\eps(t)=\int_{\mathbb{R}^n}
|\partial_tu_\eps\times u_\eps|^2(x,t)\,dx.
\]
Next we choose sufficiently small $\delta>0$ and estimate
\begin{align}\label{energy_est1}
\Big|\frac{1}{\delta}\int_0^\delta I_\eps'(t)\,dt\Big|
&\le \Big|\frac{2}{\delta}\int_0^\delta
\int_{\mathbb{R}^n}\langle (\partial_t^2u_\eps\times 
u_\eps)(x,t), (\partial_tu_\eps\times u_\eps)(x,t)-(\partial_tu_\eps\times u_\eps)(x,0)\rangle\,dxdt\Big|\nonumber\\
&\ \ +\Big|\frac{2}{\delta}\int_0^\delta
\int_{\mathbb{R}^n}\langle (\partial_t^2u_\eps\times u_\eps)(x,t), (\partial_tu_\eps\times u_\eps)(x,0)\rangle\,dxdt\Big|\nonumber\\
&=A_\eps+B_\eps.
\end{align}
From \eqref{dual_est} 
\begin{align*}
 B_\eps&=\Big|\frac{2\eps}{\delta}\int_0^\delta
\int_{\mathbb{R}^n}\langle (\partial_t^2u_\eps\times u_\eps)(x,t), (\psi\times \phi)(x)\rangle\,dxdt\Big|\\
 &\le  
2\eps\sup_{0\le t\le\delta}
\Big|\int_{\mathbb{R}^n}\langle (\partial_t^2u_\eps\times u_\eps)(x,t), (\psi\times \phi)(x)\rangle\,dx\Big|\\
&\le C\eps^3\big(\|\nabla(\psi\times\phi)\|_{L^2(\mathbb{R}^n)}
+\|\psi\times\phi\|_{L^\infty(\mathbb{R}^n)}\big)\\
&\le C\eps^3
\Big(\|\psi\|_{L^\infty(\mathbb{R}^n)}
+\|\nabla\psi\|_{L^2(\mathbb{R}^n)}
+\|\psi\|_{L^\infty(\mathbb{R}^n)}
\|\nabla\phi\|_{L^2(\mathbb{R}^n)}\Big)\\
&\le C\eps^3.
\end{align*}
As for $A_\eps$, we can estimate as follows
\begin{align*}
A_\eps   
&=\Big|\frac{2}{\delta}\int_0^\delta
\int_{\mathbb{R}^m}\langle (\partial_t^2u_\eps\times 
u_\eps)(x,t), \int_0^t \partial_t(\partial_tu_\eps\times u_\eps)(x, \tau)\,d\tau\rangle\,dxdt\Big|\\
&=
\Big|\frac{2}{\delta}\int_0^\delta
\int_{\mathbb{R}^n}\langle (\partial_t^2u_\eps\times 
u_\eps)(x,t), \int_0^t (\partial_t^2u_\eps\times u_\eps)(x, \tau)\,d\tau\rangle\,dxdt\Big|\\
&\le \frac{2}{\delta}
\Big(\int_0^\delta
\int_{\mathbb{R}^n}|\partial_t^2u_\eps(x,t)|^2\,dxdt\Big)^\frac12
\Big(\int_0^\delta
\int_{\mathbb{R}^n}t\int_0^t|\partial_t^2u_\eps|^2(x,\tau)\,d\tau\,dxdt\Big)^\frac12\\
&\le 2\int_0^\delta
\int_{\mathbb{R}^n}|\partial_t^2u_\eps(x,t)|^2\,dxdt
\to 0,\ \ {\rm{as}}\ \ \delta\to 0.
\end{align*}
Substituting the estimates for $A_\eps$
and $B_\eps$ into \eqref{energy_est1} and sending $\delta\to 0$ yields
\eqref{I'-est}.
It is readily seen that \eqref{sharp_est}
follows from \eqref{I-est}, \eqref{H-est},
and \eqref{I'-est}.
\end{proof}

\section{Proof of Theorem \ref{main} for $SO(m)$-target manifold}

In this section, we will indicate the necessary changes from the proof presented in Section 2 in order to show Theorem \ref{main} for $N=SO(m)$.

First, recall the special orthogonal group of order $m\ge 1$, $SO(m)$, is defined by
\[
SO(m)=\Big\{U\in \R^{m\times m}: \ U^TU=\mathbb{I}_m, 
\ {\rm{det}}(U)=1\Big\}. 
\]
The corresponding Lie algebra of $SO(m)$, {\it so}($m$), is given by 
\[
{\it so}(m)=\Big\{U\in \R^{m\times m}: \ U^T+U=0\Big\}.
\]
Both $SO(m)$ and ${\it so}(m)$ are equipped with the matrix norm:
\[
|A|=\Big({\rm{tr}}(A^TA)\Big)^\frac12=\Big(\sum_{i, j=1}^{m}A_{ij}^2\Big)^\frac12.
\]

\subsection{$\mathcal{H}_{\phi,\psi}^\eps$ is non-empty}
Define $$U_\eps(x,t)=\exp_{\phi(x)}(t\eps\psi(x),
\ \forall(x,t)\in \R^n\times\R_+.$$
Then $U_\eps$ is an element of $\mathcal{H}_{\phi,\eps\psi}$.
Observe that we can rewrite $U_\eps$ as
\[
U_\eps(x,t)=\phi(x)\exp_e(t\eps \Omega(x)), \ {\rm{with}}\ 
\Omega(x)=\phi^{-1}(x)\psi(x).
\]
Then it is easy to see $U_\eps(x,0)=\phi(x)$ and
\begin{align}\label{t-derivative}
\partial_t U_\eps(x,t)=\eps\phi(x)\Omega(x) \exp_e(t\eps \Omega(x))
=\eps\psi(x)\exp_e(t\eps \Omega(x)),
\end{align}
so that $\partial_t U_\eps(x,0)=\eps \phi(x)\Omega(x)=\eps\psi(x)$.
Next, we want to estimate $I_\eps(U_\eps)\le C\eps^2$.
This, in particular, implies that $U_\eps\in \mathcal{H}_{\phi,\eps\psi}$.

Taking $\partial_t$ of \eqref{t-derivative}, we obtain
\begin{align}\label{tt-derivative}
\partial_t^2 U_\eps(x,t)=\eps^2\psi(x)\Omega(x)\exp_e\big(t\eps \Omega(x)\big).
\end{align}
Taking $x$-derivative of $U_\eps$, we obtain
\begin{align}\label{x-derivative}
\nabla U_\eps(x,t)
=\nabla\phi(x)\exp_e(t\eps\Omega(x))
+\phi(x)\exp_e(t\eps\Omega(x))(t\eps\nabla\Omega(x)).
\end{align}
It follows from \eqref{tt-derivative} and \eqref{x-derivative}
that
\[
\big|\partial_t^2 U_\eps(x,t)\big|\le \eps^2|\psi(x)|^2,
\ \big|\nabla U_\eps(x,t)\big|\le \Big(|\nabla\phi|+Ct\eps\big(|\nabla\psi|+
|\nabla\phi||\psi|\big)\Big)(x,t).
\]
This yields
\begin{align}\label{energyupbound}
I_\eps(U_\eps)
&\le \int_{\R^n} \big(\eps^2|\nabla\phi|^2
+C\eps^3(|\nabla\phi||\nabla\psi|+|\nabla\phi|^2|\psi|)
+C\eps^4(|\psi|^4+|\nabla\psi|^2+|\nabla\phi|^2|\psi|^2)\,dxdt\nonumber\\
&\le \eps^2\int_{\R^n}|\nabla\phi|^2
+C\eps^3,
\end{align}
with $C$ depends on $\|\nabla\phi\|_{L^2(\R^n)}$, 
$\|\nabla\psi\|_{L^2(\R^n)}$,
and $\|\psi\|_{(L^\infty\cap L^4)(\R^n)}$.

\smallskip
It follows from the calculation above that both Lemma \ref{nonempty} and Lemma \ref{existenemin} remain true. 
Since Lemma \ref{energyestimate} is obtained through inner variations in the $t$-direction, it remains true.

Therefore, it suffices to show any weak limiting map $V$
solves the wave map equation along with the Cauchy data. 

\subsection{Passing to the limit} 
Observe that if $U_\eps\in\mathcal{H}_{\phi, \eps\psi}$ 
minimizes $I_\eps$, then $V_\eps(x,t)=U_\eps(x, \frac{t}{\eps})$
would minimize $E_\eps$ over $\mathcal{H}_{\phi,\psi}$,
that is, $(V_\eps, \partial_t V_\eps)\big|_{t=0}=(\phi,\psi)$. Then we can derive the Euler-Lagrange equation 
for $V_\eps$ as follows.

For small $\delta\in \R$ and $\phi\in C^\infty_0(\R^{n+1}_+, {\it so}(m))$, consider
$V_\eps^\delta(x,t)=(V_\eps\exp_e(\delta\phi))(x,t)$.
Then it is easy to check that $V_\eps^\delta\in\mathcal{H}$
and satisfies the boundary condition $(V_\eps^\delta, \partial_t V_\eps^\delta)\big|_{t=0}=(\phi,\psi)$.
Hence, by the minimality of $U_\eps$, we have 
\[
E_\eps(V_\eps)\le E_\eps(V_\eps^\delta)
\]
so that
\[
\frac{d}{d\delta}\big|_{\delta=0}E_\eps(V_\eps^\delta)=0.
\]
By direct calculations, using $\phi^T+\phi=0$ and the fact
\[
\frac{d}{d\delta}\big|_{\delta=0}V_\eps^\delta=V_\eps\phi,
\]
we obtain that
\begin{align*}
0&=\iint_{\R^{n+1}_+}e^{-t/\eps}\Big(\eps^2\langle \partial_t^2V_\eps, \partial_t^2(V_\eps\phi)\rangle 
+\langle\nabla V_\eps, \nabla(V_\eps\phi)\rangle\Big)\,dxdt\\
&=\iint_{\R^{n+1}_+}e^{-t/\eps}\Big(\eps^2\langle \partial_t^2 V_\eps, 2\partial_tV_\eps\partial_t\phi+V_\eps
\partial_t^2\phi\rangle +\langle\nabla V_\eps, V_\eps\nabla \phi)\rangle\Big)\,dxdt,
\end{align*}
where we have used the skew-symmetry of $\phi$ that guarantees
\[
\langle \partial_t^2 U_\eps, \partial_t^2U_\eps\phi\rangle=0\
\ {\rm{and}}\ \ \langle\nabla V_\eps, \nabla V_\eps\phi\rangle =0.
\]
Now, we substitute $\phi=e^{t/\eps}\psi$ with 
$\psi\in C^\infty_0(\R^{n+1}_+, {\it so}(m))$ to obtain
\begin{align*}
\iint_{\R^{n+1}_+}\Big(\big\langle \partial_t^2V_\eps, 2\partial_tV_\eps\big(\eps^2\partial_t\psi+{\eps}\psi\big)+V_\eps\big({\eps}\partial_t\psi+\eps^2\partial_t^2\psi  +\psi\big)\big\rangle +\big\langle\nabla V_\eps, V_\eps\nabla \psi\big\rangle\Big)\,dxdt=0. 
\end{align*}
This is equivalent to say that $V_\eps$ is a weak solution of
\begin{align}\label{eps-wave1}
&\partial_t(V_\eps^T\partial_tV_\eps) -{\rm{div}}(V_\eps^T\nabla V_\eps)=G_\epsilon,
\end{align}
where
\begin{align}\label{eps-wave2}
G_\eps:=-(\eps^2\partial_t^2-\eps\partial_t)(V_\eps^T\partial_t^2V_\eps)
+2\eps^2\partial_t(\partial_t V_\eps^T\partial_t^2V_\eps)
-2\eps\partial_tV_\eps^T\partial_t^2V_\eps.
\end{align}

Since $V_\eps$ satisfies the energy bounds similar to \eqref{H1-est}, \eqref{t-derivate-est}, and \eqref{tt-derivative-est}, we have
\begin{align*}
\int_{0}^T\int_{\R^n} (|\partial_t V_\eps|^2+|\nabla V_\eps|^2)\,dxdt \le CT,
\end{align*}
and
\begin{align*}
   \iint_{\R^n\times\R_+}|\partial_t^2 V_\eps|^2\,dxdt\le C\eps^{-1}.
\end{align*}
Hence, after passing to a subsequence, we may assume that
there exists a map $V\in H^1_{\rm{loc}}(\R^n\times (0,\infty), SO(m))$ such that
$$V_\eps\rightarrow V \ \ {\rm{in}}\ \ L^2_{\rm{loc}}(\R^n\times (0,\infty)),$$
and
$$
(\partial_t V_\eps, \ \nabla V_\eps)\rightharpoonup 
(\partial_t V, \ \nabla V)\ \ {\rm{in}}\ \ L^2(\R^n\times (0,\infty)).
$$

It is not hard to verify that for
any $\phi\in C_0^(\R^n\times (0,\infty), {\it so}(m))$,
\begin{align*}
\iint_{\R^n\times (0,\infty)}
\langle G_\eps, \phi\rangle\,dxdt
&=-\iint_{\R^n\times (0,\infty)}\langle V_\eps^T\partial_t^2V_\eps,
(\eps^2\partial_t^2+\eps\partial_t)\phi\rangle\,dxdt\\
&\quad-2\iint_{\R^n\times (0,\infty)}\langle\partial_tV_\eps^T\partial_t^2V_\eps, (\eps^2\partial_t+2\eps)\phi\rangle\,dxdt.
\end{align*}
Hence
\begin{align*}
\Big|\iint_{\R^n\times (0,\infty)}
\langle G_\eps, \phi\rangle\,dxdt\Big|
&\le C\eps \Big(1+\big\|\partial_t V_\eps\big\|_{L^2({\rm{supp}}\phi)}\Big)\big\|\partial_t^2V_\eps\big\|_{L^2(\R^n\times (0,\infty))}\\
&\le C\eps^{\frac12}\to 0.
\end{align*}
While
\begin{align*}
&\iint_{\R^n\times(0,\infty)}\langle\partial_t(V_\eps^T\partial_tV_\eps) -{\rm{div}}(V_\eps^T\nabla V_\eps),\phi\rangle\,dxdt\\
&=-\iint_{\R^n\times(0,\infty)}\big(\langle V_\eps^T\partial_tV_\eps,\partial_t\phi\rangle -\langle V_\eps^T\nabla V_\eps,\nabla\phi\rangle\big)\,dxdt\\
&\rightarrow -\iint_{\R^n\times(0,\infty)}\big(\langle V^T\partial_tV,\partial_t\phi\rangle -\langle V^T\nabla V,\phi\rangle\big)\,dxdt
\end{align*}
Putting these estimates together, we obtain
\begin{align}\label{wavemap-so(m)1}
 \iint_{\R^{n}\times (0,\infty)} \Big(\langle\nabla V, V\nabla\psi\rangle-\langle \partial_t V, V\partial_t\psi\rangle\Big)\,dxdt =0.  
\end{align}
Or equivalently, $V$ is a weak solution of the wave map into $SO(m)$:
\begin{align}\label{wavemap-so(m)2}
\partial_t(V^T\partial_t V)-\nabla\cdot(V^T\nabla V)=0 
\ \ {\rm{in}}\ \ \R^{n}\times (0,\infty).    
\end{align}

It is easy to see that $V$ satisfies the trace condition
$V(x,0)=\phi(x)$. To show $\partial_t V(x,0)=\psi(x)$, we need
the following estimate on $V_\eps^T \partial_t^2 V_\eps$, stated for $U_\eps^T\partial_t^2U_\eps$ similar to Lemma \ref{dualestimate}. 

\begin{lemma}\label{dualestimate2}
Assume $\phi\in \dot{H}^1(\mathbb R^n,SO(m))$ 
and $\psi\in ({H}^1\cap L^\infty)(\mathbb R^n, \phi^*T(SO(m)))$.
Let $U_\eps$ be a minimizer of $I_\eps$ in $\mathcal{H}_{\phi,\eps\psi}$. There exists  $C>0$ such that for any $h\in L^\infty(\mathbb{R}^n, {\it so}(m))\cap \dot{H}^1(\mathbb{R}^n, {\it so}(m))$,
it holds that for a.e. $t>0$, 
\begin{align}\label{dual_est2}
 \Big|\int_{\mathbb{R}^n} 
 \langle\partial_t^2U_\eps(x,t), U_\eps(x,t) h(x)\rangle\,dx\Big|\le Ce^t\eps^2\Big(\|h\|_{L^\infty(\mathbb {R}^n)}+\|\nabla h\|_{L^2(\mathbb{R}^n)}\Big).
\end{align}    
\end{lemma} 
\begin{proof} It is similar to that of Lemma \ref{dualestimate}. Here we give a sketch. First, as in the previous subsection, we know that $U_\eps$ satisfies
\begin{align*}
0&=\iint_{\mathbb{R}^n\times\R_+}
\Big[\langle e^{-t}\partial_t^2U_\eps, \partial_t^2(U_\eps \eta_\delta h)\rangle +\eps^2 \langle e^{-t}\nabla U_\eps,  \nabla(U_\eps\eta_\delta h)\rangle \Big]\,dxdt\\
&=\iint_{\mathbb{R}^n\times\R_+}
\Big[\langle e^{-t}\partial_t^2U_\eps, (U_\eps \eta_\delta''+2\partial_tU_\eps\eta_\delta')h\rangle +
\eps^2 \langle e^{-t}\nabla U_\eps, U_\eps \nabla h\rangle \eta_\delta\Big]\,dxdt,
\end{align*}
where $h\in C_0^\infty(\R^n, {\it so}(m))$ and $\eta_\delta\in C^{1,1}(\R)$ is given by Lemma \ref{dualestimate}.

As in Lemma \ref{dualestimate}, we can bound
\begin{align*}
\iint_{\R^n\times (0,\infty)}
e^{-t}\langle \partial_t^2 U_\eps, U_\eps \eta_\delta''h\rangle\,dxdt
&=\frac{2}{\delta}\int_{T}^{T+\delta}\int_{\R^n} 
e^{-t}\langle \partial_t^2 U_\eps, U_\eps h\rangle \,dxdt\\
&\rightarrow  2e^{-T}
\int_{\R^n} \langle\partial_t^2 U_\eps, U_\eps h\rangle\,dx, \  {\rm{as}}\ \delta\to 0,
\end{align*}
\begin{align*}
&\Big|\iint_{\R^n\times (0,\infty)}
e^{-t}\langle \partial_t^2 U_\eps, 2\partial_t U_\eps \eta_\delta'h\rangle\,dxdt\Big|\\
&=\Big|4\int_T^{T+\delta} \frac{t-T}{\delta}e^{-t}\int_{\R^n}
\langle\partial_t^2 U_\eps, \partial_t U_\epsilon h\rangle\,dxdt
+4\int_{T+\delta}^{\infty}e^{-t}\int_{\R^n}\langle\partial_t^2 U_\eps, \partial_t U_\eps h\rangle \,dxdt\Big|\\
&\le C\big\|h\big\|_{L^\infty(\R^n)}
\Big(\iint_{\R^n\times (0,\infty)} e^{-t}|\partial_t^2 U_\eps|^2\,dxdt\Big)^\frac12
\Big(\iint_{\R^n\times (0,\infty)} e^{-t}|\partial_t U_\eps|^2\,dxdt\Big)^\frac12\\
&\le C\eps^2 \big\|h\big\|_{L^\infty(\R^n)},
\end{align*}
and
\begin{align*}
&\lim_{\delta\to 0}
\Big|\iint_{\R^n\times (0,\infty)} e^{-t}\langle\nabla U_\eps, U_\eps\nabla h\rangle\,dxdt\Big|\\
&=2\Big|\iint_{\R^n\times (0,\infty)}e^{-t}(t-T)^+\langle\nabla U_\eps, U_\eps\nabla h\rangle\,dxdt\Big|\\
&\le C\big\|\nabla h\big\|_{L^2(\R^n)} \sup_{k\ge 0}
\Big(\iint_{\R^n\times [T+k, T+k+1]}|\nabla U_\eps|^2\,dxdt\Big)^\frac12\\
&\le C\big\|\nabla h\big\|_{L^2(\R^n)}.
\end{align*}
Putting all these estimates together yields \eqref{dualestimate2}.
\end{proof}

Note that translating back to $V_\eps$, 
\eqref{dualestimate} yields that for a.e. $t>0$,
\begin{align}\label{dual_est3}
 \Big|\int_{\mathbb{R}^n} 
 \langle\partial_t^2V_\eps(x,t), V_\eps(x,t) h(x)\rangle\,dx\Big|\le Ce^t\Big(\|h\|_{L^\infty(\mathbb {R}^n)}+\|\nabla h\|_{L^2(\mathbb{R}^n)}\Big).
\end{align}  

With the help of \eqref{dual_est3}, we can sketch the proof of $\partial_t V(x,0)=\psi(x)$ for $x\in\R^n$ as follows. 
For $\eta\in C_0(\R)$,
\begin{align*}
\int_0^\infty\eta(t)\int_{\R^n}\partial_t\langle \partial_tV_\eps, V_\eps h\rangle\,dx\,dt
=-\int_0^\infty\eta'(t)\int_{\R^n}\langle \partial_tV_\eps, V_\eps h\rangle\,dx\,dt
-\eta(0)\int_{\R^n}\langle\psi,\phi\rangle h\,dx.
\end{align*}
Since 
$$\partial_t\langle \partial_tV_\eps, V_\eps h\rangle
=\langle\partial_t^2 V_\eps, V_\eps h\rangle
+\langle\partial_tV_\eps, \partial_tV_\eps h\rangle
=\langle\partial_t^2 V_\eps, V_\eps h\rangle,
$$
it follows from \eqref{dual_est3} that we can pass to the limit $\eps\to 0$ in the above identity to achieve
\begin{align*}
\int_0^\infty\eta(t)\int_{\R^n}\partial_t\langle \partial_tV, V h\rangle\,dx\,dt
=-\int_0^\infty\eta'(t)\int_{\R^n}\langle \partial_tV, V h\rangle\,dx\,dt
-\eta(0)\int_{\R^n}\langle\psi,\phi h\rangle\,dx.
\end{align*}
On the other hand, one has
\begin{align*}
\int_0^\infty\eta(t)\int_{\R^n}\partial_t\langle \partial_tV, V h\rangle\,dx\,dt
&=-\int_0^\infty\eta'(t)\int_{\R^n}\langle \partial_tV, V h\rangle\,dx\,dt\\
&\quad-\eta(0)\int_{\R^n}\langle\partial_tV(x,0),\phi(x)h(x)\rangle\,dx.
\end{align*}
Therefore, by choosing $\eta(0)\not=0$ we obtain 
\begin{align*}
\int_{\R^n}\langle\partial_tV(x,0)-\psi(x),\phi(x)h(x)\rangle\,dx
=0.
\end{align*}
Since $\partial_t V(x,0)-\psi(x)\in T_{\phi(x)}{SO(m)}$, it follows
that $(\partial_t V(x,0)-\psi(x))^T\phi(x)$ is skew-symmetric. 
Since $h\in C_0^\infty(\R^n, {\it so}(m))$ is an arbitrary skew-symmetric matrix, we must have that 
\[
\partial_t V(x,0)-\psi(x) =0
\ \ {\rm{ for\ a.e.}}\ x\in\R^n.
\]

\medskip
Finally, as in Theorem \ref{globalenergyineq1} the global energy inequality \eqref{globalenergyineq2} can be proved 
once we show the following improved estimate, which is analogous to Lemma \ref{sharp_est0}. 

\begin{lemma}\label{energyineq2}
Assume $\phi\in \dot{H}^1(\mathbb R^n,SO(m))$ 
and $\psi\in ({H}^1\cap L^\infty)(\mathbb R^n, \phi^*T(SO(m)))$.
Let $U_\eps$ be a minimizer of $I_\eps$ in $\mathcal{H}_{\phi,\eps\psi}$.
Then \begin{align}\label{sharp_est2}
 E_\eps(0)
 =I_\eps(0)-I_\eps'(0)+\frac12 H_\eps(0)\le
 \frac12\eps^2\int_{\mathbb{R}^n}
 (|\nabla\phi|^2+|\psi|^2)(x)\,dx+C\eps^3.
\end{align}
\end{lemma}
\begin{proof} Similar to Lemma \ref{sharp_est0}. First, by definition
\[
I_\eps(0)=\frac12\int_{\R^n}\eps^2|\psi|^2\,dx,
\]
and by \eqref{energyupbound},
\[
\frac12 H_\eps(0)=\frac12 I_\eps(U_\eps)
\le \frac12(\int_{\R^n}\eps^2|\nabla\phi|^2\,dx+C\eps^3).
\]
Thus, it remains to show
\[
|I_\eps'(0)|\le C\eps^3.
\]

Note that $U_\eps:\R^n\times (0,\infty)\to SO(m)$, one has
\[
I_\eps(t)=\int_{\R^n}|U_\eps^T\partial_t U_\eps|^2\,dx
\]
Hence, for $\delta>0$ small, we can estimate
\begin{align*}
\Big|\frac{1}{\delta}\int_0^\delta I_\eps'(t)\,dt\Big|
&\le  
\Big|\frac{2}{\delta}\int_0^{\delta} \langle \partial_t(U_\eps^T\partial_tU_\eps), U_\eps^T\partial_t U_\eps-U_\eps^T(x,0)\partial_t U_\eps(x,0)\rangle\,dxdt\Big|\\
&+
\Big|\frac{2}{\delta}\int_0^{\delta} \langle \partial_t(U_\eps^T\partial_tU_\eps), U_\eps^T(x,0)\partial_t U_\eps(x,0)\rangle\,dxdt\Big|.   
\end{align*}
Now we can apply the same estimates as in Lemma \ref{sharp_est0}, with \eqref{dual_est} replaced by
\eqref{dual_est2}, to bound
\begin{align*}
&\Big|\frac{2}{\delta}\int_0^{\delta} \langle \partial_t(U_\eps^T\partial_tU_\eps), U_\eps^T(x,0)\partial_t U_\eps(x,0)\rangle\,dxdt\Big|\\
&\le C\eps^3 \big(\|\psi\|_{L^\infty(\R^n)} 
+\|\nabla\psi\|_{L^2(\R^n)}+\|\psi\|_{L^\infty(\R^n)}\|\nabla\phi\|_{L^2(\R^n)}\big)\le C\eps^3,
\end{align*}
and
\begin{align*}
&\Big|\frac{2}{\delta}\int_0^{\delta} \big\langle \partial_t(U_\eps^T\partial_tU_\eps), U_\eps^T\partial_t U_\eps-U_\eps^T(x,0)\partial_t U_\eps(x,0)\big\rangle\,dxdt\Big|\\
&\le C\int_0^{\delta}\int_{\R^n}|\partial_t^2 U_\eps|^2\,dxdt\rightarrow 0, 
\end{align*}
as $\delta\to 0$.   Putting all these estimates together yields \eqref{sharp_est2}.
\end{proof}

\bigskip
Finally, we would like to raise the question whether
Theorem \ref{main} holds for Riemannian homogeneous
spaces, or even general compact smooth target manifolds.

\bigskip
\noindent{\bf Acknowledgments}. The first author is partially supported by an AMS-Simons travel grant. The second author is partially supported by NSF DMS 2453789 and Simons Travel Grant TSM-00007723.


\begin{thebibliography}{}

\bibitem{Giorgi1996} E. De Giorgi, {\it Conjectures concerning some evolution problems}. Duke Math. J. 81 (1996), 255–268.

\bibitem{Freire1996} A. Freire, {\em Global weak solutions of the wave map system to compact homogeneous spaces.}
Manuscripta Math. 91 (1996), no. 4, 525-533.


\bibitem{FreireMullerStruwe1997} A. Freire, S. M\"uller,M. Struwe,
{\em Weak convergence of wave maps from  $(1+2)$-dimensional Minkowski space to Riemannian manifolds.}
Invent. Math. 130 (1997), no. 3, 589-617.

\bibitem{FreireMullerStruwe1998} A. Freire, S. M\"uller,M. Struwe,
{\em Weak compactness of wave maps and harmonic maps}.
Ann. Inst. H. Poincar\'e C Anal. Non Lin\'daire 15 (1998), 
no. 6, 725-754.

\bibitem{MullerStruwe1996} S. M\'uller, M. Struwe, 
{\em Global existence of wave maps in I+2 dimensions for finite
energy data}. Top. Merhods Nonlinear Anaiy.sis, Vol. 7, 1996, pp. 245-259.


\bibitem{Serra-Tilli2012} E. Serra, P. Tilli,
{\em Nonlinear wave equations as limits of convex minimization problems:
proof of a conjecture by De Giorgi}.
Ann. Math. 175 (2012), no. 3, 1551-1574.

\bibitem{Serra-Tilli2016} E. Serra, P. Tilli,
{\em A minimization approach to hyperbolic Cauchy problems}. 
J. Eur. Math. Soc. 18 (2016), no. 9, 2019–2044.

\bibitem{Shatah1988} J. Shatah,  
{\em Weak solutions and development of singularities of the 
$SU(2)$ $\sigma$-model.}
Comm. Pure Appl. Math. 41 (1988), no. 4, 459-469.

\bibitem{Shatah-Struwe} J. Shatah, M. Struwe, 
Geometric wave equations.
Courant Lect. Notes Math., 2
New York University, Courant Institute of Mathematical Sciences, New York; American Mathematical Society, Providence, RI, 1998. viii+153 pp.

\bibitem{Zhou1999} Y. Zhou, {\it Global weak solutions for 1 + 2 dimensional wave maps into homogeneous spaces}, 
Ann. Inst. H. Poincar\'e C Anal. Non Lin\'eaire 16 (1999), 
no. 4, 411–422.
\end{thebibliography}
\end{document}